\documentclass{rhumj_new}

\usepackage[utf8]{inputenc}
\usepackage{csquotes}
\usepackage[english]{babel}
\usepackage{geometry}[margin = 0.5 in]
\usepackage{amsthm}
\usepackage{blindtext}
\usepackage{hyperref}
\usepackage{dsfont}
\usepackage{bbm}
\usepackage{amssymb}
\usepackage{tabularx}
\usepackage{amsmath}
\usepackage{mathtools}
\usepackage{tikz-qtree,tikz-qtree-compat}

\usepackage[sorting=nty, backend=biber, style = numeric, citestyle=numeric, maxnames = 4]{biblatex}
\bibliography{sample}

\usepackage{subcaption}
\usepackage{graphicx}
\usepackage{caption}
\usepackage{setspace}

\usepackage[shortlabels]{enumitem}
\usepackage[many]{tcolorbox}
\usepackage{amssymb}

\allowdisplaybreaks

\usepackage{float}

\definecolor{calpolypomonagreen}{rgb}{0.12, 0.3, 0.17}

\def\FF{\mathbb{F}}

\def\NN{\mathbb{N}}

\def\QQ{\mathbb{Q}}

\def\ZZ{\mathbb{Z}}

\title{Proportion of p-adic Polynomials Which Are Irreducible}
\author{Isaac Rajagopal}
\date{October 18, 2024}

\begin{document}


\maketitle

\begin{abstract}
    We attempt to quantify the exact proportion of monic $p$-adic polynomials of degree $n$ which are irreducible. We find an exact answer to this when $n$ is prime and $p \neq n$, and also when $n = 4$ and $p \neq 2$. Our answers are rational functions in $p$. This relates to previous work done to find exact proportions of $p$-adic polynomials of degree $n$ which have $k$ roots.
\end{abstract}

\section{Introduction}

The $p$-adic integers $\ZZ_p$ are created by taking the set of all infinite series $a_0 + a_1p + a_2 p^2 + \cdots$, with addition and multiplication working like normal power series. These series converge if we use the $p$-adic valuation, where $|x| = p^{-v_p(x)}$, with $v_p(x)$ being the largest integer $k$ such that $p^k$ divides $x$. The $p$-adic numbers were invented by Kurt Hensel, who found a way to prove that many polynomial equations over the $p$-adic numbers have a solution, called Hensel's Lemma. Solving equations over the $p$-adic numbers can relate to determining whether a diophantine equation has a rational solution, using the local-global principle (\cite{gouvea}).

Proportions of $p$-adic numbers are well defined using the Haar measure; for example $1/p$ of all $p$-adic integers are congruent to 1 modulo $p$. This can be extended to polynomials with $p$-adic coefficients. We aim to quantify the exact proportion of monic $p$-adic polynomials in $\ZZ_p[x]$ of degree $n$ which are irreducible (\cite[\S 2.3]{koblitz}). Problems like this have been studied in many other papers, where an invariant of polynomials in $\ZZ_p[x]$ is created, and mathematicians try to find the proportion of these polynomials with a specific value of this invariant. Our invariant is the reducibility of a monic polynomial.

In 2004, an exact formula for the proportion of $p$-adic polynomials of degree $n$ which split into linear factors over $\QQ_p$ was found (\cite{buhler}). In 2022, this was generalized by finding a formula for the proportion of $p$-adic polynomials of degree $n$ which have a specific number of roots in $\QQ_p$, and showed that it is a rational function in $p$ (\cite{bhargava}). When $n \leq 3$, this is the same as our problem, but the two problems differ once $n \geq 4$ because a polynomial can factor over $\QQ_p$ without any linear terms. As recently as December 2023, authors have studied the proportion of polynomials $f(x)$ in $\ZZ_p[x]$ with specific splitting types. Their results imply that the proportion of not necessarily monic polynomials in $\ZZ_p[x]$ which are irreducible must be a rational function in $p$, which does not change by replacing $p$ with $1/p$ (\cite{asvin}).

All of these papers use much more advanced techniques than ours. With only elementary techniques, we solve for the proportion of degree $n$ irreducible polynomials in the case where $n$ is prime or $n = 4$. We work by counting irreducible polynomials mod $p$, $p^2$, $p^3$, ... using various properties of $p$-adic polynomials, such as Hensel's Lemma and Newton Polygons. At each step $p^k$ for $k>1$, these polynomials do not have unique factorization. Much of our work for quartics is based on counting the different factorizations of the same polynomials. In the limit as the exponent goes to infinity, we get $\ZZ_p[x]$, which itself has unique factorization using Gauss' Lemma. (\cite{gouvea})

In Section \ref{sectionpreliminary}, we discuss some general results which will be useful to us throughout this paper. This includes Theorem \ref{theoremapprox}, which approximates the proportion of irreducibles to be close to $1/n$ for any degree $n$. Then, we find exact values of this proportion for any prime value of the degree $n$ in Section \ref{sectionprimes}. However, this problem seems to become much more complicated for composite $n$. The bulk of the paper is Section \ref{sectionquartics}, where we find an exact value of this proportion for $n = 4$, as long as $p \neq 2$. 

\section{Preliminary Results}\label{sectionpreliminary}

Consider monic polynomials in $\ZZ_p[x]$ of degree $n$. These can be represented as elements of $\ZZ_p^n$, with a polynomial $x^n + c_{n-1}x^{n-1} + \cdots + c_1 x + c_0$ represented as $(c_{n-1}, ... , c_1, c_0)$. We use the Haar measure to quantify proportions of these polynomials: for example the proportion of polynomials with constant term divisible by $p$ is $\frac{1}{p}$. This set can be seen as $\ZZ_p^{n-1} \times B(0, \frac{1}{p})$, where $B(0, \frac{1}{p})$ is the open ball in $\ZZ_p$ of radius $\frac{1}{p}$ centered at 0, which has Haar measure $1 \times \frac{1}{p}$. The proportion of polynomials with all terms divisible by $p$ is $\frac{1}{p^n}$. With the terms of a polynomial being written in descending order, this set can be seen as $B(0, \frac{1}{p})^n$, which has Haar measure $ \frac{1}{p^n}$. (\cite{koblitz} \S 2.3)

We define these polynomials to be irreducible if they cannot be written as the product of two nonconstant polynomials in $\ZZ_p[x]$. By Gauss' Lemma, this is equivalent to not being able to write them as a product of two nonconstant polynomials in $\QQ_p[x]$. If a monic polynomial $f(x)$ (mod $p^i$) is irreducible, then all polynomials in $\ZZ_p[x]$ which are congruent to $f(x)$ (mod $p^i$) are irreducible. Let this polynomial, written in $\ZZ_p^{n}$ be $(c_{n-1}, ... , c_1, c_0)$. Then the open ball $\prod_{j=1}^{n} B(c_{n-j},p^{-i})$, which is all polynomials congruent to $f(x)$ (mod $p^i$), is an open set of measure $p^{-in}$ which consists solely of irreducible polynomials. We now introduce Hensel's Lemma, our main tool for showing polynomials to be reducible, and then use it to approximate the proportion of irreducible monic polynomials in $\ZZ_p[x]$.

\textbf{Hensel's Lemma}: There are two forms of this, from (\cite{gouvea}). A degree $n$ polynomial $f(x) \in \ZZ_p[x]$ is said to be Hensel if either (1) or (2) holds. A polynomial $g(x)$, in $\ZZ/(p^k \ZZ)[x]$ is said to be Hensel if all polynomials $f(x) \in \ZZ_p[x]$ congruent to $g(x) \pmod{p^k}$ are Hensel. \begin{enumerate}[label=(\arabic*)]
\item If there is a value $\alpha$ (in any extension of $\ZZ_p$ with degree less than $n$) for which $|f(\alpha)|<|f'(\alpha)|^2$, then $f(x)$ splits into distinct factors in $\ZZ_p[x]$. 
\item If a polynomial $f(x) \in \ZZ_p[x]$ splits into two factors which are relatively prime, when viewed (mod $p$), then all polynomials in $\ZZ_p[x]$ which are congruent to $f(x)$ (mod $p$) are reducible.  
\end{enumerate}

\begin{thm}\label{theoremapprox}
Let $(n,p) = 1$, and let the proportion of monic degree $n$ polynomials in $\ZZ_p[x]$ which are irreducible be $I$. Then $I = \frac{1}{n} + O(p^{-n/2})$, as $n$ varies.
\end{thm}

\begin{proof}

This formula for the number of irreducible polynomials in $\FF_p[x]$ of degree $n$, which we call $\nu(n)$, is in Dummit and Foote's \emph{Abstract Algebra} (\cite[588]{dummit}): \footnote{In stating this formula, I have introduced the M\"obius function $\mu(m)$ for a natural number $m$. Define $\mu(m) = 0$ if $m$ can be divided by the square of some prime, $\mu(m) = 1$ if $n$  is the product of an even number of distinct primes, and $\mu(m) = -1$ if $n$ is the product of an odd number of distinct primes. }

\begin{equation}\label{eqsrthesis}
\nu(n) = \frac{1}{n}\sum_{d\mid n}\mu \left(\frac{n}{d}\right)p^{d}
\end{equation}

A corollary of this formula is that $\nu(n) = \frac{1}{n} p^n + O(p^{n/2})$. 
Each degree $n$ irreducible polynomial in $\FF_p[x]$ gives an open set of measure $p^{-n}$ which is irreducible in $\ZZ_p[x]$. This gives a total measure of $\frac{1}{n} + O(p^{-n/2})$ of irreducible polynomials. 

We consider the remainder of polynomials in $\ZZ_p[x]$, which are reducible when viewed modulo $p$ in $\FF_p[x]$. By using (2) in Hensel's Lemma, every one of these polynomials in $\ZZ_p[x]$ is Hensel, unless it is congruent to a power of one irreducible polynomial in $\FF_p[x]$. The number of polynomials mod $p$ which are 
a power of one irreducible polynomial in $\FF_p[x]$ is equal to the number of irreducible polynomials with degree properly dividing $n$, which is \[\sum_{d\mid n, d<n} \nu(d) = \sum_{d\mid n, d<n} \frac{1}{d}p^d + O(p^{d/2}) = O(p^{n/2}).\]

Recall that there are a total of $p^n$ monic polynomials in $\FF_p[x]$ of degree $n$. This gives $\frac{n-1}{n}p^{n} + O(p^{n/2})$ reducible polynomials with distinct irreducible factors. Each of these gives an open set of measure $p^{-n}$ which is full of Hensel polynomials, all reducible in $\ZZ_p[x]$. This gives a total measure of $\frac{n-1}{n} + O(p^{-n/2})$ of reducible polynomials. The polynomials which we are not sure of have a total measure of $O(p^{-n/2})$.
\end{proof}

We now introduce a lemma, which will have a corollary that will simplify our problem.

\begin{lemma}\label{lemmaMap}
    The homeomorphism $\phi: \ZZ_p[x] \to \ZZ_p[x]$ defined by $ f(x) \mapsto f(x+a)$ for any $a \in \ZZ_p$ preserves irreducibility and preserves the measure of open sets in $\ZZ_p^{n}$.
\end{lemma}

\begin{proof} 
    This map clearly preserves irreducibility. This map preserves the size of open balls in $\ZZ_p^{n}$, because this map permutes the residue classes of monic degree $n$ polynomials modulo $p^k$ for each $k$. This is because if $f(x)$ is a polynomial in $\ZZ/(p^k) [x]$, let $f(x)$ be sent to another polynomial $\phi(f(x)) = f'(x)$ in $\ZZ/(p^k) [x]$. Then all polynomials $g(x)$ in $\ZZ_p[x]$ congruent to $f(x) \mod{p^k}$ are also sent by $\phi$ to polynomials $\phi(g(x)) = g'(x) \in \ZZ_p[x]$ congruent to $f'(x) \mod{p^k}$. This is a bijection, since this map has a clear inverse $\phi^{-1}: f(x) \mapsto f(x-a)$. All open balls in $\ZZ_p^{n}$ can be made up of residue classes of monic degree $n$ polynomials modulo $p^k$ for some $k$.
\end{proof}

\begin{corollary}\label{lemmatrace}
    Let $(n,p) = 1$. Let the proportion of monic $n^{th}$ degree polynomials in $\ZZ_p[x]$ which are irreducible be $I$, if it exists. Then $I$ is equal to the proportion of monic $n^{th}$ degree polynomials in $\ZZ_p[x]$ with $x^{n-1}$ term of 0 which are irreducible. 
\end{corollary}

\begin{proof}

We map $x \mapsto x+a$ for any $a \in \ZZ_p[x]$, which will not change whether any polynomials are irreducible or change the size of any open sets in $\ZZ_p^{n}$. This will map all polynomials in $\ZZ_p[x]$ with $x^{n-1}$ term of 0 to all polynomials in $\ZZ_p[x]$ with $x^{n-1}$ term of $na$, while preserving the size of all open sets. As long as $(n,p) = 1$, we can set $a = b/n$. This map tells us that the Haar measure of the set of irreducible polynomials with $x^{n-1}$ term of $b$ is the same as the Haar measure of the set of irreducible polynomials with $x^{n-1}$ term of $0$. So, irreducible polynomials are equidistributed among $x^{n-1}$ terms.
\end{proof}

From now on, we will only consider the polynomials with $x^{n-1}$ term of 0. This reduces our problem from looking at measures within $\ZZ_p^{n}$ to measures within $\ZZ_p^{n-1}$. This corollary shows $I$ is equal to the Haar measure of the set of degree $n$ polynomials with no $x^{n-1}$ term in $\ZZ_p^{n-1}$ which are irreducible. We have one more lemma which will convert (\ref{eqsrthesis}) into this form.

\begin{lemma}\label{lemmaSrThesis}
   Let $(n,p) = 1$. Let $\nu'(n)$ be the number of irreducible polynomials in $\FF_p[x]$ of degree $n$ with no $x^{n-1}$ term. Then

    \begin{equation}\label{eqsrthesis2}
    \nu'(n) = \frac{1}{np}\sum_{d\mid n}\mu \left(\frac{n}{d}\right)q^{d}
    \end{equation} 
\end{lemma}

\begin{proof}

We start with the monic polynomials $f(x)$ mod $p$, and let $\nu(n)$ of these be irreducible. By mapping $x \mapsto x+a$, we add $na$ to the $x^{n-1}$ coefficient of $f(x)$. If we do this for the $p$ different values of $a \in \FF_p$, we get all $p$ possible values of the $x^{n-1}$ coefficient of $f(x)$, since $(n,p) = 1$. This means that these $\nu(n)$ polynomials are equidistributed among $x^{n-1}$ terms, so $\nu'(n) = \frac{1}{p}\nu(n)$. (\ref{eqsrthesis}) gives a formula for $\nu(n)$. \end{proof}

To approximate $I$, we can look at $n^{th}$ degree polynomials with no $x^{n-1}$ term modulo $p^k$ for $k \in \NN$ increasing. At each $k$, we will have some irreducible polynomials and some Hensel polynomials, which create open balls which are either irreducible or reducible. The challenging polynomials are those which are neither irreducible nor Hensel. We have to lift these to $p^{k+1}$ and then see how many of them are irreducible and how many are Hensel. As $k$ increases, this method helps us to get closer and closer to the actual value of $I$, and gives us a lower and upper bound using the measure of the known irreducibles and the measure of the known reducibles. As $k \to \infty$, this method will give an exact value for $I$. This is shown in Table \ref{tablecubics} for cubics (mod $5^k$). We now define the Newton polygon, which will be useful to simplify our problem. 

For a polynomial $f(x)$, introduce a point $(i,j)$ when the coefficient in front of $x^i$ in $f(x)$ has valuation $j$. The Newton polygon is the lower convex hull of these points. \textbf{If the Newton polygon contains multiple slopes, then the polynomial is reducible. If the Newton polygon contains only one slope and only two integer points $(i,j)$, then the polynomial is irreducible.} The ambiguous case is where the Newton polygon contains one slope but more than two integer points (\cite{gouvea} \S 7.4). We now introduce a  lemma to use this.

\begin{table}[]
\centering
\begin{tabular}{|l|l|l|l|}
\hline
k & Irreducible & Neither & Hensel \\ \hline
1 & 8 & 1 & 16 \\ \hline
2 & 20 & 5 & 0 \\ \hline
3 & 20 & 5 & 100 \\ \hline
4 & 40 & 85 & 0 \\ \hline
5 & 100 & 25 & 2000 \\ \hline
6 & 100 & 525 & 0 \\ \hline
7 & 200 & 425 & 12500 \\ \hline
\end{tabular}

\caption{Number of cubic irreducible polynomials, Hensel polynomials, and polynomials which are neither (mod $5^k$). All polynomials are monic with no $x^2$ term. Each level $k$ shows 25 polynomials for each neither polynomial (mod $5^{k-1}$), since there are 25 ways to extend this polynomial by adding a multiple of $5^{k-1}$ to its linear and constant coefficients.}\label{tablecubics}
\end{table}

\begin{lemma}\label{lemmaS}

Let the proportion (Haar measure in $\ZZ_p^{n-1}$) of monic $n^{th}$ degree polynomials in $\ZZ_p[x]$ which are irreducible be $I$. Let us assume that $I$ is well defined. Let $S$ be the set of $p$-adic polynomials $f(x)$ such that their Newton Polygon is fully above the line from $(0,n)$ to $(n,0)$. Then $I$ is equal to the proportion of monic polynomials which are irreducible in the set $\ZZ_p^{n-1} \setminus S$.

\end{lemma}

\begin{proof}
Notice that $S$ are exactly the polynomials with their $i^{th}$ term $c_i$ divisible by $p^{n-i}$, or having valuation $v_p(c_i) \geq n-i$. This means that $S$ is exactly the polynomials $f(x)$ where $g(x) = \frac{f(px)}{p^n} \in \ZZ_p[x]$. This is because the map $f \mapsto f(px)/p^n$ will divide each coefficient $c_i$ by $p^{n-i}$, and these coefficients must finish with valuation $\geq 0$. 

Let $P$ be the set of irreducible polynomials in $S$. Let $P'$ be the set of irreducible polynomials outside $S$. Then $I = (\mu(P) + \mu(P'))/(\mu(S) + \mu(\ZZ_p^{n-1}\setminus S))$. It suffices to show that $\mu(P)/\mu(S) = I$, since this would imply that $I = \mu(P')/\mu(\ZZ_p^{n-1}\setminus S)$, which is what we want. 

Consider the map $\phi: S \to \ZZ_p^{n-1}$ defined by $\phi(f(x)) = f(px)/p^n$. The map $\phi$ has an inverse $\phi^{-1}(g(x)) = p^n g(x/p)$. So $\phi$ is a bijection because of the existence of an inverse. By the definition of $S$, $\phi(f(x)) \in \ZZ_p[x]$. We will now show that $\phi$ preserves irreducibility. Using Gauss' Lemma, we look at irreducibility in $\QQ_p[x]$. Let $g(x) = g_1(x)g_2(x)$. Then $\phi^{-1}(g(x)) = p^n g(x/p) = p^{n}g_1(x/p)g_2(x/p)$. Let $f(x) = f_1(x)f_2(x)$. Then $\phi(f(x)) = p^{-n} f(px) = p^{-n} f_1(px)f_2(px)$. So $\phi$ preserves irreducibility.

Now we will show that $\phi$ preserves the measures of open balls proportionally. By looking at the requirement for each coefficient, $\mu(S)$ consists of all polynomials with $i^{th}$ coefficient $c_i$ divisible by $p^{n-i}$. Within $\ZZ_p^{n-1}$, $\mu(S) = (p^{-2})(p^{-3}) \cdots (p^{-n}) = p^{-\frac{n(n+1)-2}{2}}$. Now we will see how the size of an open set changes under $\phi$, using $\phi^{-1}$. Consider the product of open balls in $\ZZ_p^{n-1}$ written as \[ O = \prod_{i = 1}^{n-2} B(y_i, p^{-r_i}).\] This is a basis for all open sets in the product topology on $\ZZ_p^{n-1}$. Using the Haar measure, $\mu(O) = p^{-k}$, where $k = \sum_{i=1}^{n-2} r_i$. Going coefficient by coefficient, we can see how $\phi^{-1}$ acts on $O$, multiplying $c_i$ by $p^{n-i}$. This gives \[\phi^{-1}(O) = \prod_{i = 1}^{n-2} B(p^{n-i} y_i, p^{-r_i-(n-i)}).\] Using the Haar measure, $\mu(\phi^{-1}(O)) = p^{-j}$, where $j = \sum_{i=1}^{n-2} (r_i+ n-i) = k + \frac{n(n+1)-2}{2}$. So $\mu(\phi^{-1}(O))/\mu(S) = p^{-k} = \mu(O)/\mu(\ZZ_p^{n-1})$. This means that $\phi$ preserves the proportional measures of each open ball. Since all open sets can be constructed as a union of open balls, $\phi$ preserves the proportional measure of all open sets. Since $\phi$ preserves irreducibility, this gives that $\mu(P)/\mu(S) = I$, which implies what we want.
\end{proof}

\section{Prime Degree Polynomials}\label{sectionprimes}

\begin{thm}\label{theoremprimes}
    Let the proportion of polynomials in $\ZZ_p[x]$ of degree $r$ with $r$ prime, $r \neq p$ which are irreducible be $I$. Then  \[ I = \frac{(p^r-p)(1-p^{-(r+1)/2}) + r(p-1) (1 - p^{-(r+1)(r-1)/2})}{rp^r (1-p^{-r(r+1)/2 + 1})(1-p^{-(r+1)/2})}.\]
\end{thm}

\begin{proof}

Using Corollary \ref{lemmatrace}, we will only need to consider polynomials $g(x) \in \ZZ_p[x]$ with no $x^{r-1}$ term, and we will find which are irreducible. We will start by considering the mod $p$ reduction of $g(x)$ into $\FF_p[x]$. We will show that we only need to consider polynomials which are either irreducible when reduced to $\FF_p[x]$, or congruent to $x^r$ when reduced to $\FF_p[x]$. In other words, we need to show that the only reducible polynomial in $\FF_p[x]$ which can be the reduction of an irreducible $g(x) \in \ZZ_p[x]$ is $x^r$.

All of the reducible polynomials $f(x)$ in $\FF_p [x]$ satisfy form 2 of Hensel's Lemma by having relatively prime polynomial factors in $\FF_p [x]$, unless they only have one irreducible factor $\pi(x)$ in $\FF_p [x]$. In this case, $f(x) = \pi(x)^e$, with $e>1$ and $\deg \pi(x) = d$. Since $de = r$ with $r$ prime, $d = 1$. Without loss of generality, let $\pi(x)$ be monic, so $\pi(x) = x+a$ for $a \in \FF_p$, $f(x) = (x+a)^r$. The $x^{r-1}$ term of this gives that $ra \equiv 0 \pmod{p}$. Since $(p,r) = 1$, $a \equiv 0 \pmod{p}$. So, the only case where (2) in Hensel's Lemma does not apply is when $f(x) = x^r$. In all other cases, any polynomial in $\ZZ_p[x]$ which reduces to $f(x)$ is reducible by Hensel's Lemma.

First, consider polynomials in $\ZZ_p[x]$ which reduce to irreducible polynomials in $\FF_p[x]$. Consider the monic polynomials $f(x) \in \FF_p [x]$  with no $x^{r-1}$ term. By Lemma \ref{lemmaSrThesis}, $\frac{1}{pr} (p^{r}-p)$ of these are irreducible. Each of these creates an irreducible ball of size $p^{-r+1}$ by only considering terms after $x^{r-1}$. This gives a total measure within $\ZZ_p^{r-1}$ of $\frac{p^r-p}{rp^r}$, of irreducible polynomials which reduce to irreducible polynomials in $\FF_p [x]$.

Now, consider polynomials $f(x) \in \ZZ_p[x]$ which reduce to $x^r$ in $\ZZ_p[x]$. We look at the Newton polygons of these polynomials. Each Newton polygon will start at $(0,e)$, for $e>0$ being the valuation of the constant term, and end at $(r,0)$. By Lemma \ref{lemmaS}, we only need to consider all polynomials with Newton polygon below the line from $(0,r)$ to $(r,0)$ at some point. Any irreducible polynomial with Newton polygon below this line at some point must have its whole Newton polygon below this line, since the Newton polygon can only have one slope, and it ends at $(r,0)$. Therefore, we consider all irreducible polynomials with $0<e<r$.

The Newton polygons of these polynomials must have only one slope, since they are irreducible. They start at $(0,e)$, and end at at $(r,0)$. This line has no integer points between its endpoints, so there can never be any ambiguous cases. The set of polynomials with Newton polygon having one slope from $(0,e)$  to $(r,0)$ is all polynomials with $x^i$ term having valuation above $e - ie/r$ for $0 \leq i \leq r-2$ and constant term with valuation exactly $e$. The measure of the set of polynomials with this Newton polygon is $((p-1)p^{-e-1})p^{-\alpha}$, where $\alpha$ is the number of integer points $(i,j)$ with $0 \leq j \leq e - ie/r$, with $1 \leq i \leq r-2$. This is because each point below this line multiplies the measure of this set by $p^{-1}$, except when $i = 0$, where the measure of all of $\ZZ_p$ with valuation exactly $e$ is $(p-1)p^{-e-1}$. 

Pick's Theorem says that $re/2 = J + \frac{r + e + 1}{2} - 1$, where $J$ is the number of interior lattice points of the triangle with corners at $(0,0)$, $(0,e)$, $(r,0)$. So $J = \frac{re-r-e+1}{2}$. Combining these with the points with $j = 0$ and $1 \leq i \leq r-2$, we get that $\alpha = \frac{re-r-e+1}{2} + r-2$. So the measure of the set of polynomials with Newton polygon equal to $y = e - xe/r $ is \[(p-1)p^{-e-1-\frac{re-r-e+1}{2} - (r-2)} = (p-1)p^{-\frac{re+r+e-1}{2}}.\] Summing over all $e$, the measure of all of these polynomials is $(p-1) \sum_{e=1}^{r-1} p^{-(re+r+e-1)/2}$.

Including polynomials not congruent to $x^r \pmod{p}$, the measure of all of the irreducible polynomials which have Newton polygon below the line from $(0,r)$ to $(r,0)$ at some point is \[\frac{p^r-p}{rp^r} + (p-1) \sum_{e=1}^{r-1} p^{-(re+r+e-1)/2}.\] To get a formula for $I$, we must divide this by the total measure of all monic polynomials which have Newton polygon below the line from $(0,r)$ to $(r,0)$ at some point, according to Lemma \ref{lemmaS}. Considering only polynomials with no $x^{r-1}$ term, and using the method of counting points below this line to calculate measure, this is $1 - p^{-r(r+1)/2 + 1}$.

We can now simplify our formula for $I$:
\begin{align*}
I &= \frac{\frac{p^r-p}{rp^r} + (p-1) \sum_{e=1}^{r-1} p^{-(re+r+e-1)/2}}{1-p^{-r(r+1)/2 + 1}} \\
&= \frac{\frac{p^r-p}{rp^r} + (p-1) p^{(1-r)/2} p^{-(r+1)/2} \sum_{e=0}^{r-2} (p^{-(r+1)/2})^{e}}{1-p^{-r(r+1)/2 + 1}} \\
&= \frac{\frac{p^r-p}{rp^r} + (p-1) p^{-r} \frac{1 - p^{-(r+1)(r-1)/2}}{1-p^{-(r+1)/2}} }{1-p^{-r(r+1)/2 + 1}}  \\
&= \frac{(p^r-p)(1-p^{-(r+1)/2}) + r(p-1) (1 - p^{-(r+1)(r-1)/2})}{rp^r (1-p^{-r(r+1)/2 + 1})(1-p^{-(r+1)/2})}
\end{align*}

\end{proof}

We wonder whether this result holds for $r = p$. This has been proven to hold for $p = r = 2$ and $p = r = 3$ (\cite{bhargava}). 

\section{Proportion of Quartics That Are Irreducible}\label{sectionquartics}

\begin{thm}\label{theoremquartics}
    Let $p \neq 2$. Let the proportion of monic polynomials in $\ZZ_p[x]$ of degree $4$ which are irreducible be $I$. Then  \[ I = \frac{p^{12}+p^{11}+2p^{10}+4p^9 - 2p^8 + p^7 + 3p^6 - 6p^5 + 6p^4 - 4p^3 - 2p^2 - 4}{4(p+1)(p^2+1)(p^9-1)}.\]
\end{thm}

To prove this, we will have to prove two main lemmas about specific sets of polynomials which are hard to classify, Lemma \ref{lemmaDoubleQUadratics} and Lemma \ref{lemmaSlope0.5}. We defer the proofs of many claims within these lemmas to Appendix \ref{appendix}. These two lemmas are about cases where a quartic can be factored as a product of two quadratic factors which are congruent $\pmod{p}$, and these require a lot of work to classify because the Newton polygon is no longer conclusive. After these lemmas are proven, the remaining work will be identical to Section \ref{sectionprimes}, using arguments of $\FF_p[x]$ and Newton Polygons to complete the calculations. Throughout this whole section, we consider monic polynomials with no $x^4$ term, following Corollary \ref{lemmatrace}.

\begin{lemma}\label{lemmaDoubleQUadratics}

Let $p$ prime, $p \neq 2$. Let $\pi(x) = x^2 + \zeta$ in $\ZZ_p[x]$ be irreducible when viewed modulo $p$, so $-\zeta$ is not a square modulo $p$. Let $f(x) = \pi(x)^2 = (x^2 + \zeta)^2$. Let the proportion of monic $4^{th}$ degree polynomials in $\ZZ_p[x]$ congruent to $f(x) \pmod{p}$ which are irreducible be $I$. Then \[I = \frac{2p^2 + 1}{2(p^2 + 1)}.\]

\end{lemma}

\begin{table}[]
\centering
\begin{tabular}{|l|l|l|l|}
\hline
k & Irreducible & Ambiguous & Hensel \\ \hline
1 & 0 & 1 & 0 \\ \hline
2 & 120 & 5 & 0 \\ \hline
3 & 60 & 5 & 60 \\ \hline
4 & 120 & 5 & 0 \\ \hline
5 & 60 & 5 & 60 \\ \hline
6 & 120 & 5 & 0 \\ \hline
7 & 60 & 5 & 60 \\ \hline
\end{tabular}

\caption{Proportions of quartic irreducible polynomials, Hensel polynomials, and polynomials which are neither (listed as ambiguous) (mod $5^k$), congruent to $(x^2 + 2)^2 \pmod{5}$. All polynomials are monic with $x^3$ term of $0$, so there are three coefficients which vary. Each level $k$ shows 125 polynomials for each neither polynomial (mod $5^{k-1}$), since there are 125 ways to lift this polynomial to (mod $5^k$) by adding a multiple of $5^{k-1}$ to any of its quadratic, linear, and constant coefficients.}\label{tabledoublequadratics}
\end{table}

We will prove this by showing that these polynomials all create tables that look like Table \ref{tabledoublequadratics}. Then, we can tally up the measure of all irreducible balls of polynomials using a geometric series.

\begin{proof}

We consider the reducible polynomials congruent to $f(x) \pmod{p}$ in $\ZZ_p[x]$. Since factorizations are unique in $\FF_p[x]$, any factorization of one of these polynomials in $\ZZ_p[x]$ must be congruent to $(x^2+\zeta)(x^2 + \zeta) \pmod{p}$. We can write these factorizations as $(x^2 + p\alpha x + p\beta + \zeta)(x^2 + p\delta x + p\gamma + \zeta)$. Using Corollary \ref{lemmatrace}, we only need to consider the polynomials with $x^3$ term of $0$. This gives that $\delta = -\alpha$. So we are considering all polynomials that can be produced by factorizations of the form \[f^{(\alpha,\beta,\gamma)}(x) = (x^2 + p\alpha x + p\beta + \zeta)(x^2 -p \alpha x + p\gamma + \zeta), \alpha, \beta, \gamma \in \ZZ_p.\]

We now see when we can use Hensel's Lemma on this factorization to create an open ball of reducible polynomials. Let $p^k$ be the maximal power of $p$ dividing $\alpha$ and $\beta-\gamma$. This claim is proved in the Appendix \ref{appendixclaim1}.

\begin{claim}\label{claim1}
    If $p^k \mid \alpha, \beta-\gamma$, and $p^{k+1} \nmid \alpha$ or $p^{k+1} \nmid \beta-\gamma$, then all polynomials congruent to $f(x) = f^{(\alpha,\beta,\gamma)} (x) \pmod{p^{2k+3}}$ are reducible, because $f(x)$ satisfies Hensel's Lemma $\pmod{p^{2k+3}}$.
\end{claim}

We now notice an algebraic identity which can change our factorization into a difference of squares: \[f^{(\alpha,\beta,\gamma)} (x) = (x^2 + p\alpha x + p\beta + \zeta)(x^2 - p\alpha x + p \gamma + \zeta) = \left(x^2 + p \frac{\beta + \gamma}{2}+ \zeta\right)^2 - p^2 \left(\alpha x + \frac{\beta-\gamma}{2}\right)^2.\] Let $A = \frac{\beta + \gamma}{2}$, $B = \alpha$, and $C = \frac{\beta-\gamma}{2}$. Then $A$, $B$, and $C$ are in $\ZZ_p$ and are in $1:1$ correspondence with $\alpha$, $\beta$, and $\gamma$ in $\ZZ_p$, using that $\beta = A + C$ and $\gamma = A-C$. Then we want to consider all polynomials $f_{(A,B,C)}$ of the form

\begin{equation}\label{fABC}
    f_{(A,B,C)}(x) = (x^2 + pA + \zeta)^2-p^2(Bx + C)^2, A, B, C \in \ZZ_p.
\end{equation}

\textbf{For the rest of this lemma, let $k$ represent the largest power of $p$ that divides $B$ and $C$. (We will use $j$ for arbitrary exponents.)} From Claim \ref{claim1}, $f_{(A,B,C)}(x)$ is Hensel modulo $p^{2k+3}$. Also, $f_{(A,B,C)}(x) \equiv (x^2 + pA + \zeta)^2 \pmod{p^j}$ for any $j<2k+3$, but not for $j = 2k+3$. So, the lowest power of $p$ where we will find difference between $f_{(A,B,C)}(x)$ and $f_{(A,0,0)}(x) = (x^2 + pA + \zeta)^2$ is $p^{2k+3}$.

Our idea is to count how many polynomials $f_{(A,B,C)}(x)$ there are modulo $p^j$ for any $j$. To do this, we start with a polynomial of the form $f_{(A,0,0)}(x) = (x^2 + pA + \zeta)^2 \pmod{p^{j-1}}$. When this polynomial is lifted to $\pmod{p^j}$, there are $p^3$ polynomials $g(x) \in \ZZ/(p^{j})[x]$ (monic with no $x^3$ term) that are congruent to $f_{(A,0,0)}(x) \pmod{p^{j-1}}$. There are then three possibilities that can happen: 1) $g(x) = f_{(A',0,0)}(x)$; 2) $g(x) = f_{(A',B',C')}(x)$; 3) $g(x)$ is irreducible. In 1), we must continue by lifting $g(x)$ up to $\pmod{p^{j+1}}$. In 2), by the work of the previous paragraph, $j = 2k+3$, so $g(x)$ will satisfy Hensel's lemma. So all polynomials congruent to $g(x) \pmod{p^j}$ will be reducible. Also, since $j = 2k+3$, $j$ is odd. Once we can count the number of polynomials in 1) and 2), we will then be able to count 3), the irreducible polynomials, by taking the complement. We refer to polynomials in 1) as ``Ambiguous'', in 2) as Hensel, and in 3) as irreducible.

By starting with $j = 1$, and repeating this process of lifting all Ambiguous polynomials, removing all Hensel polynomials, and counting all irreducible polynomials, we can count the number of irreducible polynomials (mod $p^j$) which are congruent to $(x^2 + \zeta)^2$, which is the goal of the lemma. 

Counting all Ambiguous polynomials (mod $p^j$) is easy: There are exactly $p^{j-1}$ of them, because there are $p^{j-1}$ choices for $pA \pmod{p^j}$, which determine $(x^2 + pA + \zeta)^2 \pmod{p^j}$. We now count Hensel polynomials.

We begin with a claim that if $f_{(A,B,C)}(x) = f_{(A',B',C')}(x)$ (mod $p^{2k+3}$), then $f_{(A,0,0)}(x) \equiv f_{(A',0,0)}(x) \pmod{p^{2k+3}}$. This reduces our problem of counting all Hensel polynomials $f_{(A,B,C)}(x)$ to counting the number of polynomials of the form $f_{(A,0,0)}(x) = (x^2 + pA + \zeta)^2$ and $p^2(Bx+C)^2$ separately, since we can treat the first and second terms of Equation (\ref{fABC}) as their own problems. This claim is proved in Appendix \ref{appendixclaim2}.

\begin{claim}\label{claim2}
    Let $(x^2 + pA + \zeta)^2 - p^2(Bx+C)^2 \equiv (x^2 + pA' + \zeta)^2 - p^2(B'x+C')^2 \pmod{p^{2k+3}}$, with $A, B, C \in \ZZ_p$, $p^k \mid B, C, B', C'$. Then $A \equiv A' \pmod{p^{2k+2}}$, so $f_{(A,0,0)}(x) \equiv f_{(A',0,0)}(x) \pmod{p^{2k+3}}$.
\end{claim}

Modulo $p^{2k+3}$, there are $p^{2k+2}$ Ambiguous polynomials $f_{(A,0,0)}(x)$. We now count polynomials of the form $p^2(Bx+C)^2 \pmod{p^{2k+3}}$.

Set $B = bp^k$, $C = cp^k$, with $b,c \in \ZZ_p$. Then $p^2(Bx+C)^2 = p^{2k+2} (bx+c)^2$. Since we are looking modulo $p^{2k+3}$, consider $b$ and $c$ modulo $p$, in $\FF_p$. By looking at quadratic, linear, and constant terms, $(bx+c)^2 = (b'x + c')^2$ if and only if $b' = \pm b$ and $c' = \pm c$, and these two signs are the same. So the only case where the map from $b$ and $c$ in $\FF_p$ to $(bx+c)^2$ is not injective is where $(bx+c)^2 = (-bx-c)^2$. The only case where this is not a duplication is when $b = c = 0$. 

This gives us $(p^2-1)/2$ nonzero values of $p^2(Bx+C)^2$ (mod $p^{2k+3}$), all of which have $p^{k+1} \nmid B$ or $p^{k+1} \nmid C$. Therefore, for each polynomial $f_{(A,0,0)}(x)$ (mod $p^{2k+3}$), there are $(p^2-1)/2$ Hensel polynomials of the form $f_{(A,B,C)}(x)$ modulo $p^{2k+3}$, by varying $B$ and $C$. In other words, we are partitioning the polynomials of the form $f_{(A,B,C)}(x)$ modulo $p^{2k+3}$ based on possible values of the first term $f_{(A,0,0)}(x) = (x^2 + pA + \zeta)^2$, and the second term $p^2(Bx+C)^2$ of Equation (\ref{fABC}).

We can now calculate the proportion of monic polynomials in $\ZZ_p[x]$ congruent to $(x^2 + \zeta)^2 \pmod{p}$ which are irreducible, $I$. Consider the $p^3$ polynomials formed by lifting an Ambiguous polynomial $m(x) = f_{(A,0,0)}(x)$ from modulo $p^{j-1}$ to modulo $p^{j}$. There are $p$ Ambiguous polynomials $f_{(A',0,0)}(x)$ modulo $p^{j}$ which are congruent to $m(x)$ modulo $p^{j-1}$. If $j$ is even, the remaining $p^3-p$ polynomials are all irreducible, since polynomials can only become Hensel for odd $j$, since they become Hensel when $j = 2k+3$. If $j$ is odd, half of the remaining $p^3-p$ polynomials are irreducible, while the other half are reducible and Hensel. This is because each of the $p$ Ambiguous polynomials $f_{(A',0,0)}(x)$ modulo $p^{j}$ leads to $(p^2-1)/2$ Hensel polynomials $f_{(A',B',C')}(x)$ modulo $p^{j}$ by varying $B'$ and $C'$. This is shown in Table \ref{tabledoublequadratics}. Each irreducible mod $p^{j}$ creates an irreducible ball of measure $p^{-3j}$ in $\ZZ_p$. The measure of monic polynomials in $\ZZ_p[x]$ congruent to $(x^2 + \zeta)^2 \pmod{p}$ for a specific $\zeta$ is $p^{-3}$. Dividing out by this, each irreducible modulo $p^{j}$ contributes $p^{-3j+3}$ to $I$. 

This gives us an infinite series for $I$, where we start with the $p^{j-2}$ polynomials of the form $f_{(A,0,0)}(x) = (x^2 + pA + \zeta)^2$ (mod $p^{j-1}$), and lift them each to $p^3-p$ irreducible polynomials (mod $p^{j}$) for $j$ even, and $(p^3-p)/2$ irreducible polynomials (mod $p^{j}$) for $j$ odd.

\begin{align*}
    I &= \sum_{j \geq 2 \; \text{even}} p^{j-2} (p^3 - p) p^{-3j+3} +  \sum_{j \geq 2 \; \text{odd}} p^{j-2} \frac{p^3 - p}{2} p^{-3j+3}\\
    &= \sum_{j \geq 2 \; \text{even}} p^{-2j} p^2(p^2 - 1) +  \sum_{j \geq 2 \; \text{odd}} p^{-2j} \frac{p^2(p^2 - 1)}{2} \\
    &= p^2(p^2 - 1) p^{-4} + \frac{p^2(p^2 - 1)}{2} p^{-6} + p^2(p^2 - 1)p^{-8} + \frac{p^2(p^2 - 1)}{2} p^{-10} + \cdots\\
    &= (p^2-1)\frac{2p^2 + 1}{2} p^{-4} \sum_{j=0}^{\infty} p^{-4j} \\
    &= (p^2-1)\frac{2p^2 + 1}{2} \frac{p^{-4}}{1-p^{-4}} \\
    &= \frac{2p^2+1}{2(p^2+1)}.
\end{align*}
\end{proof}

\begin{lemma}\label{lemmaSlope0.5}

Let $p$ prime, $p \neq 2$. Let $S$ be the set of quartic polynomials with a Newton Polygon of one line from $(0,2)$ to $(4,0)$, and no $x^3$ term. In other words $f(x) = x^4 + c_2 x^2 + c_1 x + c_0$ is in $S$ if $v_p(c_0) = 2$, $v_p(c_1) \geq 2$, and $v_p(c_2) \geq 1$. Let the proportion of monic $4^{th}$ degree polynomials in $\ZZ_p[x]$ in $S$ which are irreducible be $I$. Then \[I = \frac{p^2+3p+1}{2p(p + 1)} = \frac{p+2}{2(p + 1)} + \frac{1}{2p}. \]

\end{lemma}

Notice the similarity between the number in this result and the value of $I$ for quadratic polynomials from Theorem \ref{theoremprimes}, $I = (p+2)/2(p+1)$. 

\begin{proof}
Polynomials in $S$ cannot have any linear factors, since the only slope of their Newton Polygon is $-1/2$ (\cite{gouvea} \S 7.4). If they factor, they must factor as the product of two quadratics, each of the form $(x^2 + pc_1 x + pc_0)$, with $p \nmid c_0$, so that the slope of the Newton Polygon of these quadratics is $-1/2$. By looking at $x^3$ terms, the only way that polynomials in $S$ will factor is as \[f^{(\alpha,\beta,\gamma)}(x) = (x^2 + p \alpha x + p \beta)(x^2 - p \alpha x + p \gamma), p\nmid \beta, p \nmid \gamma.\]

We will now see when we can use Hensel's Lemma on this factorization to create an open ball of reducible polynomials. We split into two cases. Let $p^k$ be the maximal power of $p$ dividing $\alpha$ and $\beta-\gamma$. We have two cases, depending on the valuation of $\beta-\gamma$. In Case 1, $p^{k+1} \nmid \beta-\gamma$. In Case 2, $p^{k+1} \mid \beta-\gamma$, and $p^{k+1} \nmid \alpha$. We will use these throughout the rest of the lemma. This claim is proved in Appendix \ref{appendixclaim3}.

\begin{claim}\label{claim3}
     In Case 1, all polynomials congruent to $f(x) = f^{(\alpha,\beta,\gamma)} (x) \pmod{p^{2k+4}}$ are reducible, because $f(x)$ satisfies Hensel's Lemma $\pmod{p^{2k+4}}$. In Case 2, all polynomials congruent to $f(x) = f^{(\alpha,\beta,\gamma)} (x) \pmod{p^{2k+5}}$ are reducible, because $f(x)$ satisfies Hensel's Lemma $\pmod{p^{2k+5}}$.
\end{claim}

We now notice an algebraic identity which can change our factorization into a difference of squares: \[f^{(\alpha,\beta,\gamma)} (x) = (x^2 + p\alpha x + p\beta)(x^2 - p\alpha x + p \gamma) = \left(x^2 + p \frac{\beta + \gamma}{2}\right)^2 - p^2 \left(\alpha x + \frac{\beta-\gamma}{2}\right)^2.\] Let $A = \frac{\beta + \gamma}{2}$, $B = \alpha$, and $C = \frac{\beta-\gamma}{2}$. Then $A$, $B$, and $C$ are in $\ZZ_p$ and are in $1:1$ correspondence with $\alpha$, $\beta$, and $\gamma$ in $\ZZ_p$, using that $\beta = A + C$ and $\gamma = A-C$. The constant term of $f^{(\alpha,\beta,\gamma)} (x)$ is $p^2(A^2 - C^2)$, which we require not to be divisible by $p^3$. Then we want to consider all polynomials $f_{(A,B,C)}$ of the form

\begin{equation}\label{eqfABC}
f_{(A,B,C)}(x) = (x^2 + pA)^2-p^2(Bx + C)^2 \quad A, B, C \in \ZZ_p, p \nmid A^2 - C^2.
\end{equation}

Let $k$ be the maximal integer such that $p^k \mid B$ and $p^k\mid C$. Consider the sets \begin{align}
    \text{Case 1: } T &= \{f_{(A,B,C)}(x) \in \ZZ_p[x] : p^k \mid B,C\text{, and } p^{k+1} \nmid C\} \text{ and } \label{Case1def}\\
    \text{Case 2: } U &= \{f_{(A,B,C)}(x) \in \ZZ_p[x] : p^k \mid B, p^{k+1} \nmid B\text{, and } p^{k+1} \mid  C \}.\label{Case2def}
\end{align} The sets $T$ and $U$ represent cases 1 and 2, respectively. From Claim \ref{claim3}, polynomials in $T$ are Hensel modulo $p^{2k+4}$ and polynomials in $U$ are Hensel modulo $p^{2k+5}$. Now we consider all factorizations of these forms. For any $j<2k+3$, $f_{(A,B,C)}(x) \equiv (x^2 + pA)^2 \pmod{p^j}$, but not necessarily for $j = 2k+3$. So, the lowest power of $p$ where we will find a difference between $f_{(A,B,C)}(x)$ and $f_{(A,0,0)}(x) = (x^2 + pA)^2$ is modulo $p^{2k+3}$. 

We will calculate how many polynomials there are that are reducible via Case 1 modulo $p^{2k+3}$ and $p^{2k+4}$, where they will be Hensel. We will then calculate how many polynomials there are that are reducible via Case 2 modulo $p^{2k+3}$, $p^{2k+4}$, and $p^{2k+5}$ where they will be Hensel. After we complete this, we will find there to be a geometric series for the measure of all irreducible polynomials modulo $p^i$ that are congruent to polynomials which are neither irreducible nor Hensel modulo $p^{i-1}$. Tallying this up will give our final result.

Case 1: We start with a claim about there being no overlap between polynomials coming from different values of $(x^2 + pA)^2$, which is proved in Appendix \ref{appendixclaim4}.

\begin{claim}\label{claim4}
    Let $f_{(A,B,C)}(x) \equiv f_{(A',B',C')}(x) \pmod{p^{2k+4}}$, $p^k \mid  B, C, B', C'$, and $p^{k+1} \nmid C, C'$. Then $A \equiv A' \pmod{p^{2k+3}}$. Further, if we only have $f_{(A,B,C)}(x) \equiv f_{(A',B',C')}(x)$ modulo $p^{2k+3}$, then $A \equiv A' \pmod{p^{2k+2}}$.
\end{claim}

Modulo $p^i$, there are $p^{i-1}$ choices for a polynomial of the form $(x^2 + pA)^2$, since $A$ can be any number modulo $p^{i-1}$, and all of these create different $x^2$ terms. To satisfy Equation (\ref{eqfABC}), we require that $p \nmid A^2 - C^2$. When $p \nmid C$, this gives us $p-2$ choices for $A \pmod{p}$, since $A \not \equiv \pm C \pmod{p}$. This is the case where $k = 0$. When $k>0$, we have $p-1$ choices for $A \pmod p$, since we require that $p \nmid A$. These are the only requirements for $A$. Therefore the number of choices for $(x^2 + pA)^2 \pmod{p^{2k+3}}$ is $p(p-2)$ if $k = 0$, and $p^{2k+1}(p-1)$ if $k > 0$. The number of choices for $(x^2 + pA)^2 \pmod{p^{2k+4}}$ is $p^2(p-2)$ if $k = 0$, and $p^{2k+2}(p-1)$ if $k > 0$. 

We now calculate the number of choices for $p^2(Bx+C)^2 \pmod{p^{2k+3}}$ and $p^2(Bx+C)^2 \pmod{p^{2k+4}}$. By setting $B = bp^k$ and $C = cp^k$ with $p \nmid c$, we are equivalently calculating the number of choices for $(bx+c)^2 \pmod{p}$ and $(bx+c)^2 \pmod{p^2}$. Consider the choices $\pmod{p}$. We have $\frac{p-1}{2}$ choices for $c^2 \pmod{p}$ and $\frac{p+1}{2}$ choices for $b^2 \pmod{p}$. Each of the choices of $b^2$ except for $b^2 \equiv 0 \pmod{p}$ allows for a binary choice of the sign of $2bc$. This gives a total of $\frac{p-1}{2}\cdot 2 + 1 = p$ choices for $b^2$ and $2bc$. So, the total number of choices for $(bx+c)^2 \pmod{p}$ is $p(p-1)/2$. 

Now we consider $(bx+c)^2 \pmod{p^2}$. We have $\frac{p^2-p}{2}$ choices for $c^2 \pmod{p^2}$, since the group of $p^2-p$ units of $\ZZ/p^2\ZZ$ is cyclic. If $p \nmid b$, we have similarly $\frac{p^2-p}{2}$ choices for $b^2 \pmod{p^2}$, and a binary choice for the sign of $2bc$. If $p \mid b$, we have one choice for $b^2 \pmod{p^2}$, but we have $p$ choices for $2bc \pmod{p^2}$. This gives us $\frac{p^2-p}{2} \cdot 2 + p = p^2$ choices for $b^2$ and $2bc$. So, the total number of choices for $(bx+c)^2 \pmod{p^2}$ is $p^3 (p-1)/2$.

Using Claim \ref{claim4}, the choices of $(x^2 + pA)^2$ and $p^2(Bx+C)^2$  modulo $p^{2k+3}$ and $p^{2k+4}$ are independent. This allows us to give a total count of the polynomials reducible by Case 1, as defined by $T$ in (\ref{Case1def}):

\begin{equation}\label{case1}
|\{\text{Case 1 Reducible Polynomials mod $p^i$}\}| = \begin{cases} 
p^2(p-2)(p-1)/2 & \text{ if } i = 3 \\
p^5(p-2)(p-1)/2 & \text{ if } i = 4 \\
p^{2k+2}(p-1)^2/2 & \text{ if } i = 2k+3, k>0 \\
p^{2k+5}(p-1)^2/2 & \text{ if } i = 2k+4, k>0
\end{cases}
\end{equation}

These are all of the polynomials reducible by Case $1$. None of them are congruent to $(x^2 + pA)^2 \pmod{p^i}$, because of an intermediate result from Claim \ref{claim4} that $p\mid c^2 - c'^2$, which has no solutions if we set $c = 0$, and require that $p \nmid c'$. By Claim \ref{claim3}, all of these Case 1 reducible polynomials are Hensel modulo $p^{2k+4}$, and none are Hensel modulo $p^{2k+3}$. Notice that the number of Case 1 polynomials with $i = 2k+4$ is $p^3$ times the number with $i = 2k+3$, for $k \geq 0$. This means that all possible extensions of a Case 1 polynomial mod $p^{2k+3}$ extended to mod $p^{2k+4}$ remain reducible by Case 1, and satisfy Hensel's Lemma. So all polynomials congruent to a Case 1 polynomial mod $p^{2k+3}$ are reducible. We now move to consider Case 2, as defined by $U$ in (\ref{Case2def}).

First, we show that there are no polynomials in $U$ not congruent to a polynomial of the form  $(x^2 + pA)^2 \pmod{p^{2k+3}}$. Let $p^{k+1}\mid C$, and $p^{k} \mid  B$. Then 
\begin{align}
f_{(A,B,C)}(x) &= (x^2 + pA)^2-p^2(Bx + C)^2 & \nonumber\\
&\equiv x^4 + 2pA x^2 + p^2A^2 - p^2B^2x^2 &\pmod{p^{2k+3}} \nonumber\\
&\equiv \left(x^2 + p\left(A - \frac{pB^2}{2}\right)\right)^2 &\pmod{p^{2k+3}} \label{congruency}
\end{align}

This implies that all Case 2 polynomials $f_{(A,B,C)}(x)$ are congruent to a polynomial of the form  $(x^2 + pA)^2 \pmod{p^{2k+3}}$. Now we consider $\pmod{p^{2k+4}}$ and $\pmod{p^{2k+5}}$. We will have some overlap between polynomials coming from different values of $pA$  $\pmod{p^{2k+4}}$ or $\pmod{p^{2k+5}}$, meaning that $f_{(A,B,C)}(x) \equiv f_{(A',B',C')}(x)$, but $f_{(A,0,0)}(x) \not \equiv f_{(A',0,0)}(x)$. However, this overlap between different values of $A$ can be quantified. Every polynomial of the form $f_{(A,B,C)}(x)$ satisfies $f_{(A,B,C)}(x) \equiv f_{(A',B',C')}(x)$ $\pmod{p^{2k+4}}$ or $\pmod{p^{2k+5}}$ with $p$ different values of $f_{(A',0,0)}(x)$. The following claims, with Claim \ref{claim5} about $\pmod{p^{2k+4}}$ and Claim \ref{claim6} about $\pmod{p^{2k+5}}$, each have two parts and are proved in Appendix \ref{appendixclaim6}.

\begin{claim}\label{claim5}
    1) Let $f_{(A,B,C)}(x) \equiv f_{(A',B',C')}(x) \pmod{p^{2k+4}}$, $p^k \mid  B, B'$, $p^{k+1} \nmid B, B'$, and $p^{k+1} \mid C, C'$. Then $A \equiv A' \pmod{p^{2k+2}}$. 
    
    2) Now fix $A, B, C,$ and $A'$ such that $A \equiv A' \pmod{p^{2k+2}}$. Then there is exactly one possible value of $p^2(B'x + C')^2 \pmod{p^{2k+4}}$ such that $f_{(A,B,C)}(x) \equiv f_{(A',B',C')}(x) \pmod{p^{2k+4}}$. 
\end{claim}

\begin{claim}\label{claim6}
    1) Let $f_{(A,B,C)}(x) \equiv f_{(A',B',C')}(x) \pmod{p^{2k+5}}$, $p^k \mid  B, B'$, $p^{k+1} \nmid B, B'$, and $p^{k+1} \mid C, C'$. Then $A \equiv A' \pmod{p^{2k+3}}$. 
    
    2) Fix $A, B, C,$ and $A'$ such that $A \equiv A' \pmod{p^{2k+3}}$. Then there is only one possible value of $p^2(B'x + C')^2 \pmod{p^{2k+5}}$ such that $f_{(A,B,C)}(x) \equiv f_{(A',B',C')}(x) \pmod{p^{2k+5}}$. 
\end{claim}

These claims are proved in Appendix \ref{appendixclaim6}. Notice that when $f_{(A,B,C)}(x) \equiv f_{(A',B',C')}(x) \pmod{p^{2k+4}}$ or $\pmod{p^{2k+5}}$, this only implies that $A \equiv A' \pmod{p^{2k+2}}$ or $\pmod{p^{2k+3}}$, respectively. So, $(x^2 + pA)^2 \equiv (x^2 + pA')^2 \pmod{p^{2k+3}}$ or $\pmod{p^{2k+4}}$, respectively. This gives $p$ possible values of $f_{(A',0,0)}(x) = (x^2 + pA')^2 \pmod{p^{2k+4}}$ or $\pmod{p^{2k+5}}$ that can give $f_{(A,B,C)}(x) \equiv f_{(A',B',C')}(x)\pmod{p^{2k+4}}$ or $\pmod{p^{2k+5}}$.

We now use the two previous claims to count Case 2 reducible polynomials. The only requirement on $A$ to satisfy Equation (\ref{eqfABC}) is that $p \nmid A^2 - C^2$. Since $p \mid C$ in Case 2 (see (\ref{Case2def})), this gives us $p-1$ choices for $A \pmod p$. Using Claim \ref{claim5}, to count Case 2 reducible polynomials modulo $p^{2k+4}$ it suffices to count values of $A$ $\pmod{p^{2k+2}}$ which satisfy $p \nmid A$, and values of $p^2(Bx + C)^2 \pmod{p^{2k+4}}$, and multiply these numbers. Similarly, using Claim \ref{claim6} to count Case 2 reducible polynomials modulo $p^{2k+5}$, it suffices to count values of $A$ $\pmod{p^{2k+3}}$ which satisfy $p \nmid A$, and values of $p^2(Bx + C)^2 \pmod{p^{2k+5}}$, and multiply these numbers.

There are $p^{2k+1}(p-1)$ choices for $A$ $\pmod{p^{2k+2}}$ and there are $p^{2k+2}(p-1)$ choices for $A$ $\pmod{p^{2k+3}}$, since we require $p \nmid A$.

We now calculate the number of choices for $p^2(Bx+C)^2 \pmod{p^{2k+4}}$ and $p^2(Bx+C)^2 \pmod{p^{2k+5}}$ that follow the conditions of (\ref{Case2def}). By setting $B = bp^k$ and $C = cp^{k+1}$ with $p \nmid b$, we are equivalently calculating the number of choices for $(bx+pc)^2 \pmod{p^2}$ and $(bx+pc)^2 \pmod{p^3}$. For $\pmod{p^2}$, we have $\frac{p^2-p}{2}$ choices for $b^2 \pmod{p^2}$ since the group of $p^2-p$ units of $\ZZ/p^2\ZZ$ is cyclic. For each of these, we have $p$ choices for $2pbc \pmod{p^2}$. So, the total number of choices for $(bx+pc)^2 \pmod{p^2}$ is $p^2(p-1)/2$. 

Now we consider $(bx+c)^2 \pmod{p^3}$. We have $\frac{p^3-p^2}{2}$ choices for $b^2 \pmod{p^3}$, since the group of $p^3-p^2$ units of $\ZZ/p^3\ZZ$ is cyclic. For each of these, we have $p^2$ choices $2pbc \pmod{p^3}$. This choice of $2pbc$ will determine $p^2c^2 \pmod{p^3}$. So, the total number of choices for $(bx+pc)^2 \pmod{p^3}$ is $p^4 (p-1)/2$.

Using Claim \ref{claim5}, the choices of $A$ modulo $p^{2k+2}$ and $p^2(Bx+C)^2$  modulo $p^{2k+4}$ are independent. Using Claim \ref{claim6}, the choices of $A$ modulo $p^{2k+3}$ and $p^2(Bx+C)^2$  modulo $p^{2k+5}$ are also independent. This allows us to give a total count of the polynomials reducible by Case 2, as defined by $U$ in (\ref{Case2def}):

\begin{equation}\label{case2}
|\{\text{Case 2 Reducible Polynomials mod $p^i$}\}| = \begin{cases} 
p^{2k+3}(p-1)^2/2 & \text{ if } i = 2k+4, k\geq 0 \\
p^{2k+6}(p-1)^2/2 & \text{ if } i = 2k+5, k\geq 0
\end{cases}
\end{equation}

These are all of the polynomials reducible by Case $2$. None of them are congruent to $(x^2 + pA)^2 \pmod{p^i}$, because of an intermediate result from Claim \ref{claim5} and Claim \ref{claim6} that $p\mid b^2 - b'^2$, which has no solutions if we set $b = 0$, and require that $p \nmid b'$. These polynomials are all Hensel modulo $p^{2k+5}$, and none are Hensel modulo $p^{2k+4}$, using Claim \ref{claim3}. Notice that the number of Case 2 polynomials with $i = 2k+5$ is $p^3$ times the number with $i = 2k+4$, for $k \geq 0$. This means that all possible extensions of a Case 2 polynomial mod $p^{2k+4}$ extended to mod $p^{2k+5}$ remain reducible by Case 2, and satisfy Hensel's Lemma. So all polynomials congruent to a Case 2 polynomial mod $p^{2k+4}$ are reducible.

Now we confirm that there is no overlap between Case 1 and Case 2 reducible polynomials, in other words in (\ref{Case1def}) and (\ref{Case2def}), $T \cap U = \emptyset$. This follows from some of our earlier results. Consider a Case 1 polynomial $f(x)$ and a Case 2 polynomial $g(x)$ $\pmod{p^i}$, for $i \geq 4$. So $f(x)$ is a polynomial from (\ref{case1}) and $g(x)$ is a polynomial from (\ref{case2}). If $i$ is even, then $f(x) \not \equiv (x^2 + pA)^2 \pmod{p^{i-1}}$ for any $A \in \ZZ_p$, while $g(x) \equiv (x^2 + pA)^2 \pmod{p^{i-1}}$ for some $A \in \ZZ_p$, by (\ref{congruency}). If $i$ is odd, then $f(x) \equiv (x^2 + pA)^2 \pmod{p^{i-1}}$ for some $A \in \ZZ_p$, and $g(x) \not \equiv (x^2 + pA)^2 \pmod{p^{i-1}}$ for any $A \in \ZZ_p$. So $f(x) \neq g(x)$.

At last, we can combine the number of polynomials from Case 1 and 2 into a geometric series, using (\ref{case1}) and (\ref{case2}). The total measure of polynomials with Newton polygon having only slope $-1/2$ is $p^{-6}(p-1)$, since we require that $p$ divides their $x^2$ coefficient, $p^{2}$ divides their $x$ coefficient, and $p^2$ but not $p^{3}$ divides their constant coefficient. When we reduce polynomials having this Newton polygon, there are $p-1$ possible polynomials in $\ZZ/(p^2)[x] \pmod{p^2}$, namely $(x^2+pA)^2$ for $A = 1, 2, \ldots p-1$.

Consider a polynomial $h(x)$ which differs from $f_{(A,0,0)}(x) = (x^2+pA)^2 \pmod{p^i}$, with $h(x)$ reducible by Case 1 or Case 2. From the work we have done so far, any polynomial congruent to $h(x) \pmod{p^i}$ is reducible. Modulo $p^{i}$, there are $(p-1)p^{i-2}$ polynomials of the form $f_{(A,0,0)}(x)$, since $p \nmid A$. These are the only polynomials which lift to a mix of irreducible and reducible polynomials. 

Now, we lift the polynomials $f_{(A,0,0)}(x)$ from modulo $p^{i-1}$ to modulo $p^{i}$, similarly to Lemma \ref{lemmaDoubleQUadratics}. We count reducible polynomials which are congruent to a polynomial of the form $f_{(A,0,0)}(x) \pmod{p^{i-1}}$, but are not congruent to any $f_{(A',0,0)}(x) \pmod{p^{i}}$. By our previous work, these polynomials are either Case 1 or Case 2, as defined by $T$ and $U$ in (\ref{Case1def}) and (\ref{Case2def}), with $i = 2k+3$ for Case 1 and $i = 2k+4$ for Case 2. By our previous work, any polynomial congruent to them $\pmod{p^i}$ is reducible. Using (\ref{case1}) and (\ref{case2}), modulo $p^{i}$, there will be $p^{i-1}(p-1)^2/2$  if $i>3$, and $p^{i-1}(p-2)(p-1)/2$ if $i = 3$. With $i$ odd, all will be Case 1, and with $i$ even, all will be Case 2.

We now create a geometric series for $I$. Let $i \geq 3$. Modulo $p^i$, we consider all ambiguous polynomials, which are all of those congruent to some polynomial of the form $(x^2+pA)^2 \pmod{p^{i-1}}$. This is an open set in $\ZZ_p^3$, monic quartics with no trace, of size $(p-1)p^{i-3} p^{-3(i-1)} = (p-1)p^{-2i}$. All of these are irreducible, except a set of measure $(p-1)p^{i-2} p^{-3i} = (p-1)p^{-2i-2}$ which are of the form $(x^2+pA)^2 \pmod{p^{i}}$ and a set of measure $\mu = p^{-3i} p^{i-1}(p-1)^2/2 = p^{-2i-1}(p-1)^2/2$  if $i>3$, and $p^{-3i} p^{i-1}(p-2)(p-1)/2 = \mu - p^{-7}(p-1)/2$ if $i = 3$, which are Case 1 or 2. This gives a geometric series for the measure $M$ of all irreducible polynomials with Newton polygon with one slope $-1/2$.

\begin{align*}
M &= p^{-7}(p-1)/2 + \sum_{i = 3}^\infty \left( (p-1)p^{-2i} - (p-1)p^{-2i-2} - \mu \right) \\ 
  &=  p^{-7}(p-1)/2 + (p-1)\left(1-p^{-2}-p^{-1}(p-1)/2\right) \sum_{i = 3}^\infty p^{-2i} \\
  &= (p-1)p^{-6}\left(\frac{1}{2p} + \left(1-p^{-2}-p^{-1}(p-1)/2\right)\frac{1}{1-p^{-2}} \right)
\end{align*}

Since the total measure of polynomials with Newton polygon having only slope $-1/2$ is $p^{-6}(p-1)$, then $M = p^{-6}(p-1)I$. Therefore, 

\begin{align*}
  I = \frac{M}{(p-1)p^{-6}} &= \frac{1}{2p} +\left(1-p^{-2}-p^{-1}(p-1)/2\right)\frac{1}{1-p^{-2}} \\
  &= \frac{1}{2p} + \frac{p+2}{2(p+1)}. 
\end{align*}
\end{proof}

Now, we can prove Theorem \ref{theoremquartics} at last, using similar methods to Section \ref{sectionprimes}.

\begin{proof}

We start with the monic polynomials $f(x)$ mod $p$ with no $x^3$ term. Using Lemma \ref{lemmaSrThesis}, we get that $\frac{1}{4p} (p^{4}-p^2)$ of these are irreducible. Each of these creates an irreducible ball of size $p^{-3}$. This gives a total measure $M_1$ within $\ZZ_p^{3}$ of irreducible polynomials: 

\[M_1 = \frac{p^4-p^2}{4p^4}.\]

All of the reducible polynomials satisfy Hensel's Lemma by having relatively prime polynomial factors modulo $p$, unless they only have one monic irreducible factor $\pi(x)$ mod $p$. In this case, $f(x) \equiv \pi(x)^e$, with $e>1$ and $\deg \pi(x) = d$. Since $de = 4$, either $d = 1$ or $d = 2$. We use Corollary \ref{lemmatrace} to only consider polynomials with $x^3$ term of 0. 

First let $d = 2$, with $\pi(x) = x^2 + c_1x + c_0$. Then the $x^3$ term of $\pi(x)^2$ modulo $p$ is $2c_1 \equiv 0$. Since $p \neq 2$, $c_1 = 0$. So $\pi(x) = x^2 + \zeta$, where this is irreducible. This is true if and only if $-\zeta$ is not a square, which happens in $\frac{p-1}{2}$ cases modulo $p$. This gives us the result from Lemma \ref{lemmaDoubleQUadratics} on $\frac{p-1}{2}$ quartic polynomials, which each have measure $p^{-3}$. This adds a measure $M_2$ of irreducible polynomials within $\ZZ_p^{3}$:

\[M_2 = \frac{(p-1)(2p^2 + 1)}{4p^3(p^2+1)}.\]

Now let $d = 1$, with $f(x) \equiv \pi(x)^4 \pmod{p}$, for $\pi(x) = x+a$ for $a \in \FF_p$. The $x^3$ coefficient of $f(x)$ is $4a \equiv 0 \pmod{p}$. Since $p \neq 2$, $a = 0$. So, we need only to consider polynomials congruent to $x^4 \pmod{p}$ with no $x^3$ term.

Now we look at the Newton polygons of these polynomials congruent to $x^4 \pmod{p}$ with no $x^3$ term, as we did in Section \ref{sectionprimes}. Each Newton polygon will start at $(0,e)$, for $e>0$ being the valuation of the constant term, and end at at $(4,0)$. By Lemma \ref{lemmaS}, we only need to consider all polynomials with Newton polygon not fully above the line from $(0,4)$ to $(4,0)$. Any irreducible polynomial with Newton polygon below this line at some point must have its whole Newton polygon below this line, since the Newton polygon can only have one slope, and it ends at $(4,0)$. Therefore, we consider irreducible polynomials with $0<e<4$. 

The Newton polygons of these polynomials have only one slope, start at $(0,e)$, and end at at $(4,0)$. In each case, the measure of all polynomials with this Newton polygon is $(1-\frac{1}{p})p^{-l}$, where $l$ is the number of points $(i,j)$ with $0 \leq i \leq 2$ and $0 \leq j < e - \frac{e}{4}i$. This is because each of these points $(i,j)$ below the Newton polygon reduces the measure of the allowed polynomials by a factor of $p$, and we require that $p^{e+1}$ does not divide the constant term, which gives the factor of $(1- \frac{1}{p})$.

If $e = 2$, we get the case from Lemma \ref{lemmaSlope0.5}. The measure of all of these polynomials is $(p-1)p^{-6}$. By the result of Lemma \ref{lemmaSlope0.5}, these add a measure $M_3$ of irreducible polynomials within $\ZZ_p^{3}$: 

\[M_3 = \frac{(p-1)(p^2 + 3p + 1)}{2p^6(p^2+p)}.\]

If $e = 1$ or $e = 3$, all polynomials with these Newton polygons are irreducible. These add a measure $M_4$ of irreducible polynomials within $\ZZ_p^{3}$: 

\[M_4 = \frac{p-1}{p^4} + \frac{p-1}{p^9}\]

We add up $M_1 + M_2 + M_3 + M_4$ to get the measure of all irreducible polynomials with Newton polygon below the line from $(0,4)$ to $(4,0)$. To get a formula for $I$, we must divide this by the total measure of all monic polynomials which have Newton polygon below the line from $(0,4)$ to $(4,0)$ at some point, according to Lemma \ref{lemmaS}. Considering only polynomials with no $x^{3}$ term, and using the method of counting points below this line to calculate measure, this is $1 - p^{-9}$.

We can now simplify our formula for $I$:

\begin{align*}
I &= \frac{M_1 + M_2 +M_3 + M_4}{1 - p^{-9}} \\
&= \frac{\frac{p^4-p^2}{4p^4} + \frac{(p-1)(2p^2 + 1)}{4p^3(p^2+1)} +  \frac{(p-1)(p^2 + 3p + 1)}{2p^6(p^2+p)} + \frac{p-1}{p^4} + \frac{p-1}{p^9}}{1-p^{-9}} \\
&= \frac{1}{4(p^9-1)(p^2+1)(p+1)}\Bigg( p^5(p^4-p^2)(p^2+1)(p+1) + p^6(p+1)(p-1)(2p^2 + 1) \\
&+ 2p^2(p-1)(p^2 + 3p + 1)(p^2+1) + 4p^5(p^2+1)(p+1)(p-1) + 4(p^2+1)(p+1)(p-1)\Bigg) \\
&= \frac{p^{12}+p^{11}+2p^{10}+4p^9 - 2p^8 + p^7 + 3p^6 - 6p^5 + 6p^4 - 4p^3 - 2p^2 - 4}{4(p+1)(p^2+1)(p^9-1)}
\end{align*}

\end{proof}

We wonder whether this result holds for $p = 2$. It has been proven that the proportion of all (not necessarily monic) polynomials in $\ZZ_p[x]$ of any degree which are irreducible is a rational function in $p$. (\cite{asvin}) This suggests that this theorem could be true for $p=2$, where we only consider monic polynomials.

\appendix 

\section{Appendix: Proofs of Claims from Section \ref{sectionquartics}}\label{appendix}

\subsection{Proof of Claim \ref{claim1}}\label{appendixclaim1}

\begin{proof}

Let us introduce a root $y$ of $x^2 + p\alpha x + p \beta + \zeta$, so $x^2 + p\alpha x + p \beta + \zeta = (x-y)(x-\overline{y})$. We have $v_p(y) = v_p(\overline{y}) = 0$, so $\ZZ_p[y]$ is an unramified extension of degree 2, $\ZZ_{p^2}$. Then $f(x) = (x-y)(x-\overline{y})(x^2 -p \alpha x + p\gamma + \zeta)$. All polynomials $g(x)$ congruent to $f(x) \pmod{p^{2k+3}}$ will satisfy $p^{2k+3}\mid g(y)$. So $|g(y)| \leq p^{-(2k+3)}$. So it will be sufficient to show that $|g'(y)| > p^{-(k+1.5)}$, or equivalently that $v_p(g'(y)) < k + 1.5$. We want to calculate $v_p(f'(y)) \pmod{p^{2k+3}}$. 

Using the product rule, $f'(y) = (y-\overline{y})(y^2 -p \alpha y + p\gamma + \zeta) = (y-\overline{y})(-2p\alpha y +p(\gamma - \beta))$. Since $y + \overline{y} = -p \alpha$, $y-\overline{y} = 2y + p \alpha$. Since $v_p(y) = 0$, we get $v_p(y - \overline{y}) = 0$. We can write $-2\alpha = p^k A$ and $\gamma - \beta = p^k B$, where $A, B \in \ZZ_p$ and either $p \nmid A$ or $p \nmid B$. Then $v_p(f'(y)) = k+1 + v_p(Ay + B)$. Thus it will be sufficient to show that $v_p(Ay + B) = 0$. 

Assume the opposite, which is that $v_p(Ay + B) \geq 1$. Take $A, B, y$ modulo $p$, and we assume that $Ay + B \equiv 0$ modulo $p$. If $A \equiv 0$ modulo $p$, then clearly $B \equiv 0$ too, which is a contradiction since either $p \nmid A$ or $p \nmid B$. Assuming $A \not \equiv 0$, we can take $y \equiv -B/A$ modulo $p$. However, we also have that $y^2 + \zeta \equiv 0$ modulo $p$. This is a contradiction, because $x^2 + \zeta$ is irreducible in $\FF_p[x]$. Therefore $v_p(Ay + B) = 0$, and it follows that $v_p(f'(y)) = k+1$. Therefore $f$ is Hensel with the root $y$, modulo $p^{2k+3}$.

\end{proof}

\subsection{Proof of Claim \ref{claim2}}\label{appendixclaim2}
\begin{proof}
Setting $B = bp^k$, $C = cp^k$, $B' = b'p^k$, and $C' = c'p^k$ and equating coefficients, we get three equations $\pmod{p^{2k+3}}$. We start with the $x^2$ coefficient:

\begin{align}
    2pA - p^2B^2 + 2\zeta &\equiv 2pA' - p^2B'^2 + 2\zeta &\pmod{p^{2k+3}} \nonumber \\
    2p (A-A') &\equiv p^{2k+2} (b^2-b'^2) &\pmod{p^{2k+3}} \label{equation1}
\end{align}

From (\ref{equation1}), we get that $p^{2k+1} \mid  A-A'$. Therefore $p^{2k+1} \mid  A^2 - A'^2$. We now look at the constant coefficient:

\begin{align}
p^2 A^2 + 2p\zeta A + \zeta^2 - p^2 C^2 &\equiv p^2 A'^2 + 2p\zeta A' + \zeta^2 - p^2 C'^2 &\pmod{p^{2k+3}} \nonumber \\
2p \zeta (A-A') &\equiv p^{2k+2} (c^2-c'^2) &\pmod{p^{2k+3}} \label{equation2}
\end{align}

Combining (\ref{equation1}) and (\ref{equation2}), we have that 

\begin{align}
\zeta p^{2k+2} (b^2-b'^2) &\equiv p^{2k+2} (c^2-c'^2) &\pmod{p^{2k+3}} \nonumber\\
\zeta(b^2-b'^2) &\equiv (c^2-c'^2) &\pmod{p} \label{equation3}
\end{align}

Now, looking at the $x$ coefficient:

\begin{align}
-2p^2BC &\equiv -2p^2B'C' &\pmod{p^{2k+3}} \nonumber \\
-2p^{2k+2} bc &\equiv -2p^{2k+2} b'c' &\pmod{p^{2k+3}}  \nonumber \\
bc &\equiv b'c' &\pmod{p} \label{equation4}
\end{align}

So we have two equations for $b,c$ in $\FF_p:$ (\ref{equation3}) and (\ref{equation4}). Define $\chi^2 = -\zeta$. Since $x^2 + \zeta$ is irreducible in $\FF_p[x]$, then $\chi$ is not in $\FF_p$. So, we move to the degree 2 extension $\FF_p[\chi]$, which is a field. Combining our two equations, $c^2+2\chi bc - \zeta b^2 = c'^2 + 2\chi b'c' - \zeta b'^2$. Or, $(c+\chi b)^2 = (c' + \chi b')^2$. Since $\FF_p[\chi]$ is a field, $c+\chi b = \pm (c' + \chi b')$. Since $\chi$ is not in $\FF_p$, $c = c'$ and $b=b'$, or $c = -c'$ and $b = -b'$. In either case, we can use that $2p (A-A') \equiv p^{2k+2} (b^2-b'^2) \pmod{p^{2k+3}}$ to get that $A \equiv A' \pmod{p^{2k+2}}$, as we desired.
\end{proof}

\subsection{Proof of Claim \ref{claim3}}\label{appendixclaim3}
\begin{proof}

Let us introduce a root $y$ of $x^2 + p\alpha x + p \beta$, so $x^2 + p\alpha x + p \beta = (x-y)(x-\overline{y})$. Using norms of $p$-adic extensions, $v_p(y) = v_p(\overline{y}) = v_p(p \beta)/2 = 1/2$. So $\ZZ_p[y]$ is a ramified extension of degree 2. Then $f(x) = (x-y)(x-\overline{y})(x^2 -p \alpha x + p\gamma)$. We now calculate $f'(y)$. 

Using the product rule, $f'(y) = (y-\overline{y})(y^2 -p \alpha y + p\gamma) = (y-\overline{y})(-2p\alpha y +p(\gamma - \beta))$.  $y + \overline{y} = -p \alpha$, so $y-\overline{y} = 2y + p \alpha$. Since $v_p(y) = 0.5$, and $p \neq 2$ we get $v_p(y - \overline{y}) = 0.5$. Then $v_p(f'(y)) = 1.5 + v_p(-2\alpha y + (\gamma - \beta))$. In Case 1, $v_p(\gamma - \beta) = k$, and $v_p(-2\alpha y) \geq k+0.5$, so $v_p(f'(y)) = k + 1.5$ and $|f'(y)| = p^{-k-1.5}$. In Case 2, $v_p(-2\alpha y) = k+0.5$, and $v_p(\gamma - \beta) \geq k+1$, so $v_p(f'(y)) = k + 2$ and $|f'(y)| = p^{-k-2}$.

All polynomials $g(x)$ congruent to $f(x) \pmod{p^{2k+4}}$ will satisfy $p^{2k+4}\mid g(y)$. So $|g(y)| \leq p^{-(2k+4)}$. In Case 1, all of these polynomials satisfy the inequality of Hensel's Lemma: $|g(y)| < |g'(y)|^2$. In Case 2, these do not satisfy the inequality. Similarly, all polynomials $g(x)$ congruent to $f(x) \pmod{p^{2k+5}}$ will satisfy $p^{2k+5}\mid g(y)$. So $|g(y)| \leq p^{-(2k+5)}$. In Case 2, all of these polynomials satisfy Hensel's Lemma.
\end{proof}

\subsection{Proof of Claim \ref{claim4}}\label{appendixclaim4}
\begin{proof}
We start by getting equations from equating the terms of $f_{(A,B,C)}(x)$ and $f_{(A',B',C')}(x)$ mod $p^{2k+4}$:

\begin{align*}
    2pA - p^2 B^2 &\equiv 2pA' - p^2 B'^2 &\pmod{p^{2k+4}} \\
    p^2BC &\equiv p^2B'C' &\pmod{p^{2k+4}} \\
    p^2 A^2 - p^2 C^2 &\equiv p^2 A'^2 - p^2 C'^2 &\pmod{p^{2k+4}}
\end{align*}

We set $B = bp^k$, $C = cp^k$, $B' = b'p^k$, and $C' = c'p^k$, with $p \nmid c, c'$. These equations now become:

\begin{align}
    2 (A - A') &\equiv p^{2k+1} (b^2-b'^2) &\pmod{p^{2k+3}} \label{eq1} \\
    bc &\equiv b'c' &\pmod{p^{2}} \label{eq2} \\
    A^2 - A'^2 &\equiv p^{2k}(c^2 - c'^2) &\pmod{p^{2k+2}} \label{eq3}
\end{align}

From (\ref{eq1}), $p^{2k+1} \mid  A - A'$. By using a difference of squares in (\ref{eq3}), $p^{2k+1} \mid  A^2 - A'^2$, so $p \mid  c^2 - c'^2$. So $c \equiv \pm c' \pmod{p}$. Using (\ref{eq2}), $b \equiv \pm b' \pmod{p}$, with $b \equiv b'$ if $c \equiv c' \pmod{p}$ and $b \equiv -b'$ if $c \equiv -c' \pmod{p}$. Using (\ref{eq1}), $p^{2k+2} \mid  A - A'$. 

If we had only required $f_{(A,B,C)}(x) \equiv f_{(A',B',C')}(x)$ mod $p^{2k+3}$, all of the logic thus far would hold, simply by decreasing the power of $p$ by 1 in every modular constraint. This concludes the proof of the second part of the claim. We now continue with the proof of the first part of the claim.

By using a difference of squares in (\ref{eq3}), $p^{2k+2} \mid  A^2 - A'^2$, so $p^2 \mid  c^2 - c'^2$ and $p^2 \mid  (c-c')(c+c')$. Since $p \nmid c, c'$, it is not possible for $p \mid  c-c'$ and $p\mid c+c'$ to both hold. Therefore either $p^2\mid c-c'$ or $p^2\mid c + c'$. So $c \equiv \pm c' \pmod{p^2}$. Using (\ref{eq2}), since $c$ and $c'$ are not divisible by $p$, we can multiply both sides by $c^{-1} \pmod{p^2}$. This gives that $b \equiv \pm b' \pmod{p^2}$, with $b \equiv b'$ if $c \equiv c' \pmod{p^2}$ and $b \equiv -b'$ if $c \equiv -c' \pmod{p^2}$. Using (\ref{eq1}), $p^{2k+3} \mid  A - A'$.
\end{proof}

\subsection{Proof of Claim \ref{claim6}}\label{appendixclaim6}
The proof of Claim \ref{claim5} is a simpler version of the proof of Claim \ref{claim6}, so we choose to omit it. 
\begin{proof}
We start by getting equations from equating the terms of $f_{(A,B,C)}(x)$ and $f_{(A',B',C')}(x)$ mod $p^{2k+5}$, and simplifying them in a similar manner to Claim \ref{claim4} by setting $B = bp^k$, $C = cp^{k+1}$, $B' = b'p^k$, and $C' = c'p^{k+1}$, with $p \nmid b, b'$. These equations now become:

\begin{align}
    2 (A - A') &\equiv p^{2k+1} (b^2-b'^2) &\pmod{p^{2k+4}} \label{eq4} \\
    bc &\equiv b'c' &\pmod{p^{2}} \label{eq5} \\
    A^2 - A'^2 &\equiv p^{2k+2}(c^2 - c'^2) &\pmod{p^{2k+3}} \label{eq6}
\end{align}

Recall from (\ref{eqfABC}) that $p \nmid A^2 - C^2$. Since $p \mid C$, $p \nmid A$. Similarly, $p \nmid A'$. From (\ref{eq4}), $p \mid  A - A'$. Therefore, $p \nmid A + A'$.

From (\ref{eq6}), $p^{2k+2} \mid A^2 - A'^2$. Since $p \nmid A + A'$, this implies that $p^{2k+2} \mid A - A'$. From (\ref{eq4}), $p \mid b^2 - b'^2$, so $b \equiv \pm b' \pmod{p}$. From (\ref{eq5}), $c \equiv \pm c' \pmod{p}$, with the sign matching that of $b$ and $b'$. From (\ref{eq6}), $p^{2k+3}\mid A^2 - A'^2$. Since $p \nmid A + A'$, this implies that $p^{2k+3} \mid  A - A'$.

From (\ref{eq4}), $p^2 \mid  b^2 - b'^2$. Since $p \nmid b$, $p \nmid b'$, and $p\mid b-b'$ or $p\mid b+b'$, we either have $p^2 \mid  b - b'$ or $p^2 \mid  b + b'$. So $b \equiv \pm b' \pmod{p^2}$. From (\ref{eq5}), we can deduce that $c \equiv \pm c' \pmod{p^2}$, with the sign matching that of $b$ and $b'$.

Now fix $A, B, C,$ and $A'$ such that $A \equiv A' \pmod{p^{2k+3}}$. Let $A - A' = p^{2k+3} \varepsilon$ for $\varepsilon \in \ZZ_p$. We want to find a unique value of $p^2(B'x + C')^2 \pmod{p^{2k+5}}$, or equivalently a unique value of $(b'x + pc')^2 \pmod{p^3}$, that satisfies (\ref{eq4}), (\ref{eq5}), and (\ref{eq6}). We thus care only about $c'^2 \pmod p$, $b'c' \pmod{p^2}$, and $b'^2 \pmod{p^3}$. Without loss of generality, assume that $b \equiv b' \pmod{p^2}$ and $c \equiv c' \pmod{p^2}$ (the other case is where $b \equiv -b' \pmod{p^2}$ and $c \equiv -c' \pmod{p^2}$ and is similar). This fully determines $c'^2 \pmod p$ and $b'c' \pmod{p^2}$, and also means that (\ref{eq5}) and (\ref{eq6}) are automatically satisfied. 

We now focus on possible values of $b'^2 \pmod{p^3}$ which solve (\ref{eq4}). (\ref{eq4}) says that \[2 p^2 \varepsilon \equiv b^2 - b'^2 \pmod{p^3}\] Let $b' = b + p^2 \delta$, for $\delta \in \ZZ_p$. Then $b^2 - b'^2 = p^2 \delta (2b + p^2 \delta) \equiv 2b p^2 \delta \pmod{p^3}$. Dividing out by $2p^2$, we have that $\varepsilon \equiv b \delta \pmod{p}$. Since $p \nmid b$, there is a unique value of $\delta \pmod{p}$ that satisfies this. So there is a unique value of $b' = b + p^2 \delta \pmod{p^3}$ that satisfies (\ref{eq4}). So there is one value of $b'^2 \pmod{p^3}$ which solves (\ref{eq4}).
\end{proof}

\printbibliography[
heading=bibintoc,
title={Bibliography}
]

\end{document}